\documentclass[reqno]{amsart}
\usepackage[utf8]{inputenc}

\usepackage{sky}
\usepackage{graphicx}
\usepackage{color}
\usepackage{caption}
\usepackage{subcaption}

\newcommand{\vertiii}[1]{{\left\vert\kern-0.25ex\left\vert\kern-0.25ex\left\vert #1 
\right\vert\kern-0.25ex\right\vert\kern-0.25ex\right\vert}}

\title{Deconfinement for $\mathrm{SO}(3)$ Lattice Yang--Mills at Strong Coupling}

\author{Ron Nissim}
\address{Ron Nissim, Department of Mathematics, Massachusetts Institute of Technology, Cambridge, MA 02139}
\email{rnissim@mit.edu}

\date{}

\begin{document}

\begin{abstract}
    We study the physics prediction that lattice Yang--Mills theories with gauge groups which have trivial centers do not satisfy Wilson's criterion for quark confinement. Specifically, we prove that $\mrm{SO}(3)$ lattice Yang--Mills theory does not satisfy Wilson's criterion for almost every inverse coupling in a strong coupling regime.
\end{abstract}

\maketitle

\tableofcontents

\section{Introduction}

Euclidean quantum Yang--Mills theory is the mathematical framework for the gauge fields of the Standard Model of particle physics. In order to initiate a rigorous mathematical treatment of the subject, \cite{Wilson1974} introduced a lattice version of the model referred to as lattice Yang--Mills theory. The data of the model are a compact matrix Lie group $G$ known as the \textit{gauge group}, a finite subgraph $\Lambda \subset \Z^d$, and an \textit{inverse coupling strength} $\beta>0$. Given these data, there is an associated lattice Yang--Mills probability measure $\mu_{G,\Lambda,\beta}$ on the space of assignments, $Q:E_{\Lambda}^+ \to G$, of group elements to each (positively oriented) edge. We refer the reader to \cite{chatterjee2016a} for a survey of the model and the important open problems in the subject.

A fundamental physical phenomenon which motivated Wilson's introduction of lattice Yang--Mills is the problem of \textit{quark confinement}. In particular, in \cite{Wilson1974} Wilson related the potential between static quark and antiquark sources separated by a distance $R$ to the Wilson loop through the formula
\begin{equs}\label{eq:quark-antiquark-potential}
    V(R) := -\lim_{T \to \infty} \frac{1}{T} \log \, \langle W_{\ell_{R,T}} \rangle_{\mu_{G,\Lambda_\infty,\beta}},
\end{equs}
where the expectation is taken with respect to the infinite-volume lattice Yang--Mills measure, $\ell_{R,T}$ is an oriented rectangle of dimensions $R$ and $T$, and  $W_{\ell_{R,T}}$ denotes the Wilson loop observable as defined in \eqref{eq:Wilson-loop-observable}. For the model considered here, the existence of this limit in the strong coupling regime is reviewed in Subsection \ref{subsection:Ref-Pos}. Wilson's criterion for quark confinement is that $V(R) \to \infty$ as $R \to \infty$.

Many lattice Yang--Mills theories are known to satisfy a stronger form of Wilson's criterion called the \textit{area law}. A lattice Yang--Mills theory obeys an area law if there exist constants, $c_1,c_2>0$ such that for all rectangular loops $\ell$,
\begin{equs}
    |\langle W_{\ell}\rangle_{\mu_{G,\Lambda_\infty,\beta}}| \leq c_1 e^{-c_2 \mrm{area}(\ell)}.
\end{equs}
This bound implies that the quark-antiquark potential $V(R)$ grows at least linearly in $R$. In works of Osterwalder and Seiler based on cluster expansion, it was shown that for $G=\mrm{SU}(N)$ which has nontrivial center, at strong enough coupling (small enough $\beta$), the corresponding Yang--Mills model obeys an area law \cite{osterwalderseiler1978,seiler1982gauge}. Moreover in recent works the strong coupling regime for which the area law holds has been substantially expanded for gauge groups $G \in \{\mrm{SU}(N), \mrm{U}(N), \mrm{SO}(2(N-1))\}$ and $N>1$ using methods based on gauge--string duality and based on rapid mixing of Langevin dynamics \cite{CNS2025a,CNS2025b}. All of the previously mentioned groups have nontrivial centers, and this was a key assumption used in each method of proof. 

In contrast, this paper is concerned with groups with trivial center. In particular, throughout this paper the gauge group is always taken to be $G=\mrm{SO}(3)$. In this setting, the prediction from the physics literature is very different: Wilson's criterion for quark confinement is not expected to hold at all in dimensions $d \geq 3$, explaining why it is often argued that center symmetry is the mechanism for quark confinement (see, for example, the discussion in the introduction of \cite{chatterjee2021probabilistic}). Our main result verifies this prediction rigorously by showing that Wilson's criterion for confinement is not met in a strong coupling regime for $G=\mrm{SO}(3)$ lattice Yang--Mills theory. 

\begin{theorem}\label{thm:Main-theorem}
    For any $d \geq 3$, there exists $\beta^*(d)>0$ such that for almost every $\beta \in (0,\beta^*(d))$ (with respect to Lebesgue measure), the quark-antiquark potential $V(R)$ corresponding to $\mrm{SO}(3)$ lattice Yang--Mills on $\Z^d$ with inverse coupling parameter $\beta$ as defined in equation \eqref{eq:quark-antiquark-potential} is uniformly bounded. That is,
    \begin{equs}
        \sup_{R \in \N} V(R) < \infty.
    \end{equs}
\end{theorem}

\begin{remark}[Perimeter law lower bound and almost everywhere restriction]
    Unfortunately the proof we provide for Theorem \ref{thm:Main-theorem} does not seem to directly yield a uniform exponential-in-perimeter lower bound for Wilson loop observables and we leave this as a problem for further investigation. This is related to the `almost every' restriction in Theorem \ref{thm:Main-theorem} which seems to be unavoidable through the current proof. The main obstruction is the inability to rule out that the Wilson loop observables  $\langle W_{\ell_{R,T}}\rangle$, as a function of $\beta$, may have zeros in the complex plane arbitrarily close to $(0,\beta^*(d))$ as the loop size gets larger (one can rule out zeros exactly on the interval $(0,\beta^*(d))$ using reflection positivity). A common zero-free complex neighborhood, independent of the loop, would allow the same complex-analytic argument to give a perimeter law lower bound throughout the interval.
\end{remark}

\begin{remark}[Comments on prior results]
    In a paper of Glimm and Jaffe and a monograph of Seiler \cite{GlimmJaffeNonConfinement,seiler1982gauge}, a perimeter law lower bound at sufficiently strong coupling is claimed under general representation-theoretic conditions analogous to the trivial center assumption. The argument outlined there compares the cluster expansion with contributions from certain `tube' clusters. The difficulty is to make this comparison uniform in the loop size. There are order $|\ell|$ single-cube perturbations of a tube, each requiring only four additional plaquettes and having a nonzero leading Taylor term. Thus the leading-order calculation for a fixed loop does not by itself show that the tubes dominate at a coupling independent of the loop. Controlling these excitations, including possible cancellations, requires an additional argument. This issue is discussed further in Section \ref{section:comments-on-cluster-expansion}. Our proof instead combines a leading-coefficient estimate with complex analysis and reflection positivity.
\end{remark}  

\begin{remark}[Generalizations and physical interpretation]\label{rmk:physics}
    We only ever include Wilson loop observables in the fundamental representation. If one considers a more general representation for Wilson loops $\tau$, and if we let $A$ denote the adjoint representation of the gauge group $G$, then the expected criterion for deconfinement/screening with any gauge group $G$ and Wilson loop representation $\tau$ is that
    \begin{equs}
        \tau\otimes A^{\otimes n}
    \end{equs}
    contains a copy of the trivial representation for some large enough $n$. In principle, our proof should be adaptable to this more general case. The main additional difficulty would be to establish an analogue of the leading-coefficient estimate, Proposition \ref{lm:Tube-Lemma}. This representation theoretic condition has the physical interpretation that gluon pairs can be born out of a vacuum, screening the representation $\tau$.
\end{remark}

\begin{remark}[Weak coupling and relation to continuum limit]
    Theorem \ref{thm:Main-theorem} holds in a strong coupling regime. Passing such a result to the continuum would require estimates in a weak coupling regime that are uniform as the lattice spacing tends to zero. One expects the screening mechanism described in Remark \ref{rmk:physics} to persist at weak coupling, but the estimates proved here do not address that regime or the continuum limit.
\end{remark}


\textit{Organization:} In Section 2 we review basic notation and definitions for lattice Yang--Mills theory. Section 3 contains the proof of Theorem \ref{thm:Main-theorem} which is divided into three steps each given their own subsection. Step 1 uses some lemmas regarding expectation values with respect to the Haar measure on $\mrm{SO}(3)$ proven in Appendix \ref{Appendix:Expectation-Values}. In Section \ref{section:comments-on-cluster-expansion} we make some comments on why the proof of Theorem \ref{thm:Main-theorem} requires more than a leading-order cluster computation. Lastly, Section \ref{Section:large-N} has a few comments about extending the proof to gauge group $G=\mrm{SO}(2N+1)$ for $N>1$, some subtleties which arise including those concerning the large $N$ limit.\\

\textit{Proof Outline:} The proof of Theorem \ref{thm:Main-theorem} is divided into three steps. The goal of step 1 is to show that the Taylor expansion for a rectangular Wilson loop observable has a nonzero coefficient of perimeter order, and this coefficient is lower bounded exponentially in the perimeter of the loop. The argument is based on a geoemetric interpretation of the Taylor series in terms of a cluster expansion, arguing first that minimal nonzero clusters must look `tube like' away from corners of the loop through symmetry considerations, and then using a reflection positivity based argument to show that exact tube contributions which can be explictly computed serve as lower bounds to the leading coefficient. Step 2 combines this estimate with analyticity and bounds from the cluster expansion, using a complex analysis lemma to obtain a single-loop perimeter lower bound outside a set of inverse couplings of small measure. This follows from a general elementary complex analysis lemma which lower bounds the modulus of analytic functions away from its zeros given a lower bound on the modulus of zero and an upper bound on the maximum modulus in a disk. Finally, Step 3 uses reflection positivity to pass from arbitrarily large squares with such a bound to uniform boundedness of the quark-antiquark potential.\\

\noindent \textbf{Acknowledgements:} The author would like to thank Sky Cao, Cillian Doherty, Christophe Garban, Slava Rychkov, Scott Sheffield, and Oren Yakir for very helpful discussions and comments. The author would also like to thank Roland Bauerschmidt, Christophe Garban, and Tyler Helmuth for pointing the author to the relevant section of Seiler's monograph. R.N. was supported by the NSF under Grant No. GRFP-2141064.\\

\noindent \textbf{AI Statement:} ChatGPT 6 Astra was used to identify a gap in the proof in the first version of this paper and to resolve it by providing Lemma \ref{lm:square-noncancellation} and its proof. LLMs were also used for minor corrections and literature searches. The author has reviewed the entire manuscript and takes full responsibility for its contents.

\section{Definitions and notation}\label{section:Preliminaries}

In this section we recall some definitions and notation relating to the lattice Yang--Mills model. Throughout this paper we set $G= \mrm{SO}(3)$ unless otherwise specified, though some results of Subsection \ref{subsection:Ref-Pos} hold for any gauge group. We will set $\Lambda_L :=[-L,L]^d  \cap \Z^d$ and $\Lambda_{\infty}:=\Z^d$. When a definition or property holds for any $\Lambda_L$ or $\Lambda_{\infty}$ we simply use $\Lambda$ to denote the lattice. We perform the integral calculations in finite volume. The cluster-expansion bounds are uniform in $L$, and fixed Taylor coefficients stabilize once the box is sufficiently large; these facts allow us to pass to $\Lambda_\infty$.  Let $E_{\Lambda}^+$ (resp. $E_\Lambda$) be the collection of positively oriented (resp. oriented) edges in $\Lambda$. Similarly, let $\mc{P}_{\Lambda}^+$ (resp. $\mc{P}_{\Lambda}$) denote the set of positively oriented (resp. oriented) plaquettes in $\Lambda$. For oriented edges $e \in E_\Lambda$, we denote by $e^{-1}$ the oppositely oriented version of $e$. Similarly, for oriented plaquettes $p \in \mc{P}_\Lambda$, we denote by $p^{-1}$ the oppositely oriented version of $p$.

\begin{definition}[Lattice gauge configuration]\label{def:gauge-configuration}
A lattice gauge configuration $Q$ is a function $Q : E_\Lambda^+ \ra G$. We always implicitly extend $Q : E_\Lambda \ra G$ to oriented edges, by imposing that $Q_{e^{-1}} := Q_e^{-1}$ for all $e \in E_\Lambda^+$.
\end{definition}

\begin{remark}
    We can, and usually will, equivalently identify the set of gauge configurations by $G^{E_{\Lambda}^+}$.
\end{remark}

\begin{definition}[Plaquette variable]
Given a lattice gauge configuration $Q : E_\Lambda^+ \ra G$ and an oriented plaquette $p = e_1 e_2 e_3 e_4$, we define the plaquette variable (abusing notation)
\begin{equs}
Q_p := Q_{e_1} Q_{e_2} Q_{e_3} Q_{e_4}.
\end{equs}
\end{definition}

Next, we define the lattice Yang--Mills measure.

\begin{definition}[Lattice Yang--Mills]\label{def:lattice-ym}
For a finite lattice $\Lambda$, the \textit{lattice Yang--Mills measure} with Wilson action and gauge group $G$ is the measure on $G^{E_{\Lambda}^+}$ given by
\begin{equs}
    d\mu_{G,\Lambda, \beta}(Q) := \frac{1}{Z_{\Lambda, \beta}} \exp(S_{G,\Lambda,\beta}(Q)) dQ,
\end{equs}
where the action is given by
\begin{equs}
    S_{G,\Lambda,\beta}(Q):=\sum_{p \in \mc{P}_{\Lambda}^+}\beta \mrm{Re}\mathrm{Tr}(Q_p)
\end{equs}
and $dQ = \prod_{e \in E_\Lambda^+} dQ_e$, and each $dQ_e$ is an independent copy of normalized Haar measure on $G$ for each edge $e \in E_{\Lambda}^+$. Finally, the partition function is
\begin{equs}
Z_{\Lambda, \beta} :=\int_{G^{E_{\Lambda}^+}}\prod_{p \in \mc{P}_{\Lambda}^+} \exp(\beta\mrm{Re}\mathrm{Tr}(Q_p)) dQ.
\end{equs}
Since we always set $G=\mathrm{SO}(3)$, we will often use the notation $\langle \cdot \rangle_{\mu_{\Lambda, \beta}}$ or $\langle \cdot \rangle_{\Lambda, \beta}$ or even just $\langle \cdot \rangle$ to denote the expectation with respect to the probability measure $\mu_{G,\Lambda, \beta}$.
\end{definition}
\begin{remark}
    We will always work with sufficiently small $\beta$ so that the measure $\mu_{\Lambda_L,\beta}$ converges weakly as $L \to \infty$ to an infinite volume limit which we will denote by $\mu_{\Lambda_{\infty},\beta}$. This is proven both in \cite{osterwalderseiler1978} using a cluster expansion based approach, and in \cite{shen2023stochastic} using a Langevin dynamics based approach.
\end{remark}

Next, we define a collection of quantities associated to the lattice Yang--Mills measure. The first is the so-called Wilson loop observable, an observable associated to any loop in the lattice.

\begin{definition}[Loops, Loop Variables, and Wilson Loop Observables]\label{def:loops-loop variables}
We will represent loops $\ell$ in $\Lambda$ by the sequence of oriented edges $\ell = e_1 \cdots e_n$ that are traversed by $\ell$. We denote $|\ell| := n$. Let $Q : E_\Lambda^+ \ra G$ be a lattice gauge configuration. For a loop $\ell=e_1e_2\dots e_n$, let $Q_{\ell}:=Q_{e_1}Q_{e_2}\dots Q_{e_n}$ denote the corresponding loop variable. We define
\begin{equs}\label{eq:Wilson-loop-observable}
W_\ell(Q) := \tr(Q_\ell),
\end{equs}
where $\tr = \frac{1}{3}\mrm{Tr}$ is the normalized trace for $G=\mrm{SO}(3)$. For a different matrix group, the normalization is the reciprocal of the dimension of its defining representation.
We refer to $W_\ell$ as a Wilson loop observable.
\end{definition}

Whenever we talk about measurability in $G^{E_{\Lambda}^+}$, we refer to the usual Borel $\sigma$-algebra. We next define a general class of observables which includes Wilson loop observables.
\begin{definition}[Local Observables]
    A local observable $f$ is a bounded Borel function on $G^{E_{\Lambda}^+}$ for which there is a finite collection of edges $e_1,\dots,e_n$ for which $f$ is $\sigma(Q_{e_1},\dots,Q_{e_n})$-measurable. We will refer to the minimal collection of edges with this property as the edges which $f$ depends on.
\end{definition}

Finally we review a certain class of symmetries known as gauge transformations which leave the Yang--Mills measure invariant.

\begin{definition}[Gauge Transformations and Invariance]
    A \textit{gauge transformation} $\mc{G}$ is a function from
    \begin{equs}
        \mc{G}:G^{E_{\Lambda}^+} \to G^{E_{\Lambda}^+}
    \end{equs}
    specified by assigning a group element $\phi_v \in G$ to every vertex $v \in \Lambda$ and sending
    \begin{equs}
        \mc{G}:(Q_e)_{e \in E_{\Lambda}^+} \to (\phi_x Q_e \phi_y^{-1})_{e=(x,y) \in E_{\Lambda}^+}.
    \end{equs}
    We say that a local observable $f$ is \textit{gauge invariant}, if $f = f \circ \mc{G}$ for any gauge transformation $\mc{G}$.
\end{definition}

Let us recall that Wilson loop observables are gauge invariant.



\section{Proof of Theorem \ref{thm:Main-theorem}}\label{Section:Proof-of-main-thm}

\subsection{Step 1: Analysis of the cluster expansion}\label{Section:Step-1}

The aim is to bound the first nonzero Taylor coefficient of the numerator of the Wilson loop expectation. We use the standard analytic and reflection-positive properties recalled in Lemma \ref{lm:cluster-expansion} and Subsection \ref{subsection:Ref-Pos}; their proofs are independent of the leading-coefficient calculation below.

For a plaquette count $K:\mc{P}_{\Lambda_L}^{+}\to\N$,
write $|K|:=\sum_p K(p)$ and
\[
 I_\ell(K):=\int W_\ell(Q)
 \prod_{p\in\mc{P}_{\Lambda_L}^{+}}
 \mrm{Tr}(Q_p)^{K(p)}\,dQ,
 \qquad K!:=\prod_p K(p)!,
\]
and let
\begin{equs}\label{eq:Wilson-numerator}
 \mc{N}_{\ell,L}(\beta)
 :=Z_{\Lambda_L,\beta}\langle W_\ell\rangle_{\Lambda_L,\beta}
 =\sum_{j\ge0}\beta^j\sum_{|K|=j}\frac{I_\ell(K)}{K!}.
\end{equs}
Only the first nonzero coefficient of this sum will be needed; we do not need to identify every cluster contributing to it.

\begin{prop}[Leading coefficient]\label{lm:Tube-Lemma}
For each $d\ge3$ there is $C_d\ge1$ such that the following holds.
For any rectangle $\ell=\ell_{R,T}$ with $R,T\ge100$, there are
$m_\ell\in\N$ and $a_\ell>0$ satisfying
\begin{equs}\label{eq:leading-coefficient}
 4|\ell|-128\le m_\ell\le4|\ell|-16,
 \qquad a_\ell\ge C_d^{-|\ell|},
\end{equs}
and, for all sufficiently large $L$,
\[
 \mc{N}_{\ell,L}(\beta)
 =a_\ell\beta^{m_\ell}+O_{\ell,L}(\beta^{m_\ell+1}).
\]
The same $m_\ell,a_\ell$ are the first degree and coefficient of
$\langle W_\ell\rangle_{\Lambda_\infty,\beta}$.
\end{prop}

\begin{lemma}[Slab bounds]\label{lm:slab-bounds}
If $I_{\ell_{R,T}}(K)\ne0$ and $R,T\ge8$, then
$|K|\ge8\max\{R,T\}$. If $R,T\ge100$, then also
\begin{equs}\label{eq:slab-support-bound}
 |\mrm{supp}K|\ge4|\ell|-128.
\end{equs}
\end{lemma}
\begin{proof}
For an edge $e$, let $\mc{S}_e$ be the unit slab between the hyperplanes through its endpoints perpendicular to $e$. In this slab, let the transverse graph have as vertices
the normal edges and as edges the supported crossing plaquettes.
It is a subgraph of $\Z^{d-1}$. Suppose a tree component contains
exactly one prescribed matrix-entry occurrence. Condition on the other
variables; its plaquette factors are
\[
 \mrm{Tr}(U_v B_{vw}U_w^{-1}A_{vw}^{-1})^{K(p)}.
\]
Root the tree. Fixed left and right translations of the $U_v$, chosen successively along its edges, preserve product Haar measure and reduce these factors
to $\mrm{Tr}(V_vV_w^{\mathsf T})^{K(p)}$.
The prescribed occurrence remains a linear combination of entries of
one $V_v$ or its transpose; Lemma~\ref{lm:Haar-zero-integral} kills each term. In particular,
a tree component containing exactly one Wilson occurrence gives zero.

An $x_1$-slab contains two Wilson edges at transverse distance $T$.
Their components either coincide, requiring at least $T$ plaquettes,
or each contains a cycle, requiring at least $4+4$ plaquettes.
Summing the $R$ disjoint slabs gives $|K|\ge8R$; exchange the coordinates
for $|K|\ge8T$.

For \eqref{eq:slab-support-bound}, retain the horizontal-side slabs
$i=4,\ldots,R-5$ and the vertical-side slabs $j=4,\ldots,T-5$.
At each of these $|\ell|-32$ side edges, its component contains a cycle
or connects to the opposite side. Select four connected supported
plaquettes attached to the chosen edge. Their transverse vertices lie
within distance four of it.
These selections are disjoint: different parallel slabs are disjoint,
and opposite sides are more than eight units apart. An overlap between
perpendicular selections would be an $(x_1,x_2)$-plaquette. A selection
near $x_1=0$ or $x_1=R$ has every vertex in $[-4,4]$ or $[R-4,R+4]$,
respectively, so it cannot contain a retained interval $[i,i+1]$.
There are therefore at least $4(|\ell|-32)$ distinct supported plaquettes.
\end{proof}

\begin{figure}
    \centering
    \includegraphics[width=.65\textwidth]{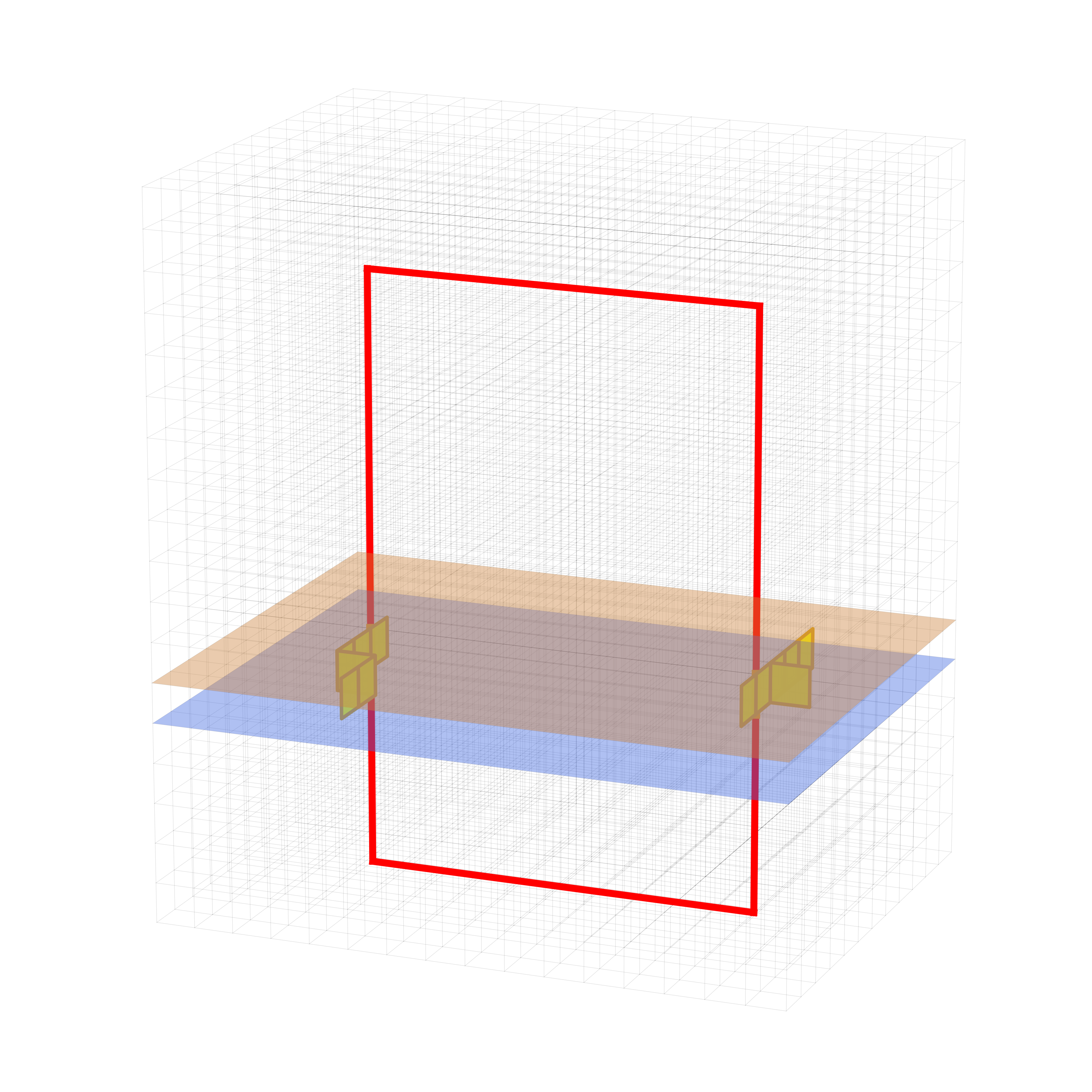}
    \caption{The loop $\ell$ is shown in red and supported plaquettes in a unit slab in yellow. In the transverse graph used in the proof, normal edges become vertices and these plaquettes become edges. A tree component containing exactly one Wilson occurrence contributes zero.}
    \label{fig:Slab-pic}
\end{figure}

\begin{lemma}[Noncancellation for even squares]\label{lm:square-noncancellation}
Let $\ell=\partial[0,n]^2\times\{0\}^{d-2}$, where $n\ge4$ is even.
Work in a translated box $\Lambda_L$ centered at
$(n/2,n/2,0,\ldots,0)$ and containing the unit frame tube around $\ell$.
For all sufficiently small positive $\beta$,
\begin{equs}\label{eq:square-numerator-lower}
 \mc{N}_{\ell,L}(\beta)
 \ge \tfrac12\,3^{-4|\ell|+15}\beta^{4|\ell|-16}.
\end{equs}
The permitted neighborhood of zero may depend on $n,L$.
\end{lemma}
\begin{proof}
\emph{The truncation inequality.}
Write $x=x_1$, $y=x_2$. Cut along $x=y$ and $x+y=n$, and let
$\mc{D}$ be the plaquettes cut diagonally. Let
$\mc{D}_1\subset\mc{D}$ consist of the four corner plaquettes
inside the square, at each of the heights $x_3=0,1$, with the remaining
coordinates zero. Thus $|\mc{D}_1|=8$. Define
\begin{equs}\label{eq:truncated-numerator}
 \widetilde{\mc{N}}(\beta)
 :=\beta^8\int W_\ell
 \prod_{p\in\mc{D}_1}\mrm{Tr}(Q_p)
 \exp\!\left(\beta\sum_{p\notin\mc{D}}
                     \mrm{Tr}(Q_p)\right)dQ.
\end{equs}
We claim that
\begin{equs}\label{eq:numerator-comparison}
 \mc{N}_{\ell,L}(\beta)\ge\widetilde{\mc{N}}(\beta)\ge0
 \qquad(\beta\ge0).
\end{equs}
Figure~\ref{fig:square-cut} records the geometric decomposition used in the
proof.

\begin{figure}
    \centering    \includegraphics[width=.72\textwidth]{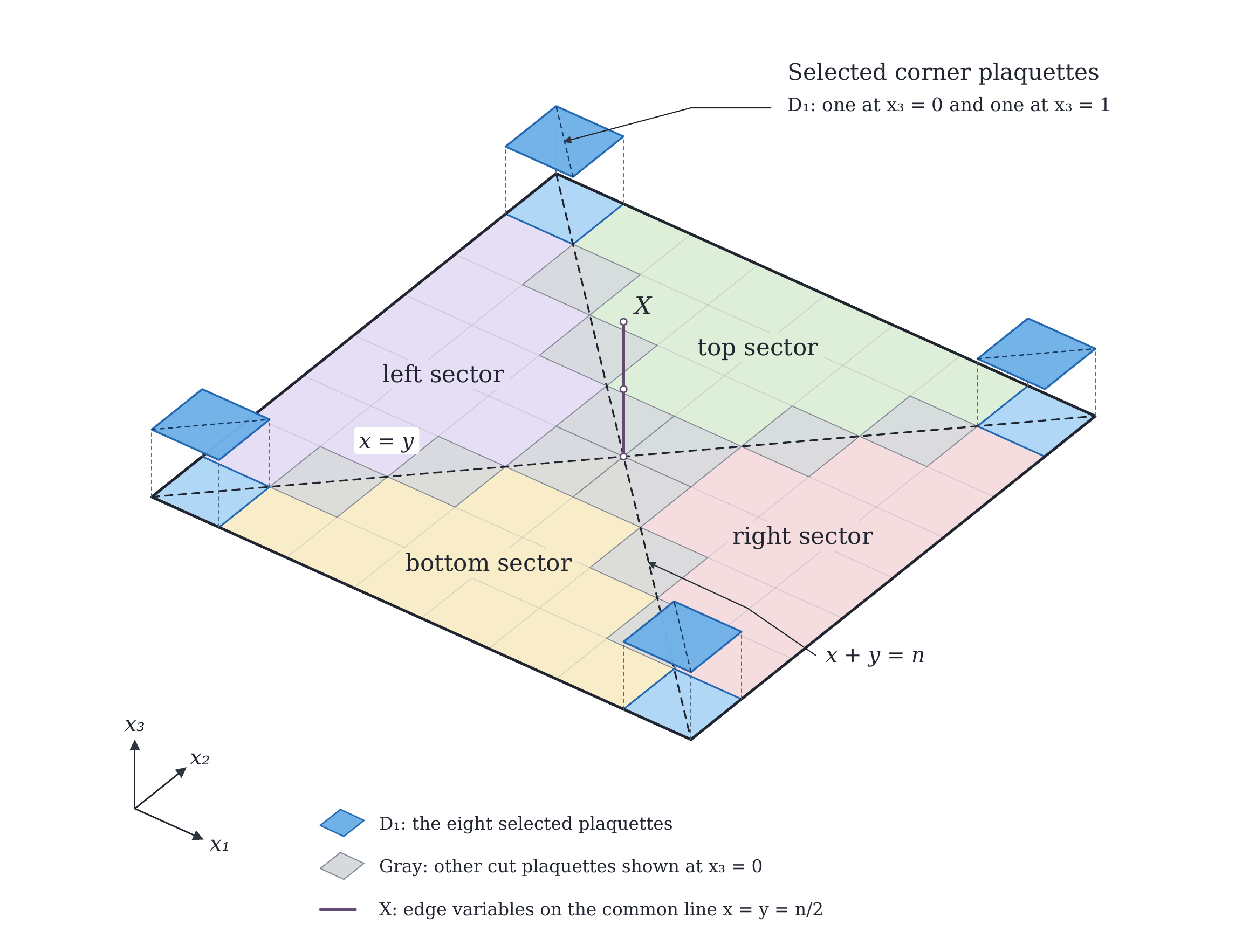}
    \caption{The even square is cut along the diagonals $x=y$ and $x+y=n$.
    The dashed lines determine the four sectors. The set $\mc{D}$ consists of
    the plaquettes cut by these diagonals, and $\mc{D}_1\subset\mc{D}$ is the
    set of selected corner plaquettes used in the truncation. The common
    intersection of the four sectors is the shared core $X$. At each corner,
    the picture shows one selected plaquette; in the proof there is one such
    plaquette at height $x_3=0$ and one at height $x_3=1$.}
    \label{fig:square-cut}
\end{figure}

Orient both halves of a cut plaquette toward the center. If their
holonomies are $U,V$, its weight is
\begin{equs}\label{eq:cut-plaquette-expansion}
 e^{\beta\mrm{Tr}(UV^{-1})}
 =\sum_{k\ge0}\ \sum_{I,J\in\{1,2,3\}^k}
 f_{k,I,J}(U)f_{k,I,J}(V),
 \qquad
 f_{k,I,J}(U):=\frac{\beta^{k/2}}{\sqrt{k!}}(U^{\otimes k})_{IJ}.
\end{equs}
The degree-zero factor is $1$.
Condition on the edge variables $X$ in the common intersection
$\{x=y=n/2\}$. Since $n$ is even, no plaquette is cut through its
interior by both diagonals. All physical edges on an interface are
transverse edges. Split the action of an interface-contained plaquette
equally between its neighboring sectors, and write $\rho_\beta(X)$
for the weight of plaquettes in the common intersection.

For a fixed value of $X$, let $\Omega$ denote the space of all possible
states on one sector interface. Thus a point $u\in\Omega$ consists of three
pieces of data:
\begin{itemize}
    \item the physical edge variables on that interface, outside the common
    intersection $X$;
    \item for each cut plaquette meeting the interface, one label $(k,I,J)$
    from \eqref{eq:cut-plaquette-expansion};
    \item one Wilson matrix index in $[3]$.
\end{itemize}
So $\Omega$ is a product of a compact group-variable space with a discrete
index set. Integration over $\Omega$ means Haar measure in the physical edge
variables and counting measure in the discrete indices. The bottom-sector
integral is
\begin{equs}\label{eq:sector-integral}
 A_\beta(X;u,v)
 =\int (Q_\gamma)_{ab}\,e^{\beta V_0}
       \prod_{p\text{ cut along the sector boundary}}
          f_{\alpha_p}(Q_{\gamma_p})\,dQ_{\mathrm{int}}.
\end{equs}
The bottom sector is $\{y\le x,\ x+y\le n\}$. Here $\gamma$ is the bottom
Wilson side, $\gamma_p$ is the half-path of $p$ in this sector, and
$V_0$ sums its uncut plaquette traces,
with weight $1/2$ on its interfaces and with the common intersection
excluded. The indices $a,b,\alpha_p$ are components of $u,v$.

Identify interfaces by the quarter-turn $r(x,y)=(n-y,x)$, fixing the
other coordinates. Under this identification, the sector integral is real and
symmetric in $u,v$: reflection $(x,y)\mapsto(n-x,y)$ interchanges the two
interfaces and transposes $Q_\gamma$, and on those interfaces it agrees with
$r$ and $r^{-1}$. Contracting the four sectors therefore gives the following
sum of squares, where $A=A_\beta(X;\cdot,\cdot)$:
\begin{equs}\label{eq:numerator-squares}
 \mc{N}_{\ell,L}(\beta)
 =\frac13\int\rho_\beta(X)
   \int_{\Omega^2}\left(\int_\Omega A(u,v)A(v,w)\,dv\right)^2
   du\,dw\,dX.
\end{equs}
All sums and integrals converge absolutely by Cauchy--Schwarz, since
$\sum_{k,I,J}|f_{k,I,J}(U)|^2=e^{3\beta}$ and the other factors are bounded.

Let $E\subset\Omega$ be the subset in which the two selected cut
plaquettes on an interface have degree $1$ and all other cut plaquettes have
degree $0$; the physical variables and Wilson indices remain unrestricted.
Restricting each $\Omega$ integration in \eqref{eq:numerator-squares} to $E$
produces exactly $\widetilde{\mc{N}}$.
For the comparison, first retain only $u,w\in E$, discarding nonnegative
squares. Then split the inner integral into $v\in E$ and $v\in E^c$, and
expand the square. The $E$-term is exactly $\widetilde{\mc{N}}$, the
$E^c$-term is nonnegative, and by symmetry and Fubini the integrated cross
term is
\begin{equs}\label{eq:cross-term-squares}
 2\int_{E\times E^c}
    \left(\int_E A(u,v)A(u,z)\,du\right)^2dv\,dz\ge0.
\end{equs}
Discarding these nonnegative terms and integrating against
$\rho_\beta(X)/3$ proves \eqref{eq:numerator-comparison}.

\emph{First nonzero term of the truncation.}
Fix $X$ and $u,v\in E$, and expand the uncut plaquette weights in
\eqref{eq:sector-integral}. The four selected half-plaquette factors
contribute $\beta^2$. At the left end their paths are
\[
 (0,0,j)\longrightarrow(1,0,j)\longrightarrow(1,1,j),
 \qquad j=0,1,
\]
where $j=x_3$ and the remaining coordinates are zero. The right-end paths
are their reflected copies.

In each interior slab $i\le x\le i+1$, $1\le i\le n-2$, the Wilson
occurrence is the only $x$-edge occurrence outside the full plaquette
factors. All $x$-edge variables are integrated, since no such edge lies
in a sector interface. The conditional tree argument in
Lemma~\ref{lm:slab-bounds} therefore forces a cycle in the transverse
graph component containing the vertex representing the Wilson edge.
Thus each interior slab requires at least four distinct full plaquettes.
Each end slab needs
one further full plaquette containing the upper prescribed $x$-edge;
otherwise this edge occurs once and its integral vanishes by
Lemma~\ref{lm:Haar-SO(3)-Moments}. These disjoint requirements give
$4(n-2)+2=4n-6$ full plaquette factors, counted with multiplicity.
Thus the total degree is at least $4n-4$.

At equality, each interior slab contains exactly four multiplicity-one
plaquettes, forming the four $x$-parallel faces of a unit cube. Each
end slab contains one plaquette, and there are no other positive-degree
factors. The left end plaquette must join the prescribed $x$-edges at
heights $0$ and $1$: any other choice leaves its second $x$-edge with
one occurrence. It is therefore $[0,1]\times\{0\}\times[0,1]$.

At $x=1$, this end plaquette and the two selected half-paths supply three
sides of the square $\{1\}\times[0,1]\times[0,1]$. These three edges are
integrated variables and each occurs once so far. They must therefore
belong to the square boundary of the first interior cube, which fixes
that cube uniquely. At the shared face $x=j$ of consecutive interior
slabs, $2\le j\le n-2$, all four edges of the propagated square lie
strictly inside the sector, since $y\le1<\min\{j,n-j\}$. The next cube
must have the same square boundary; otherwise an integrated edge occurs
once. This forces the cubes
$[i,i+1]\times[0,1]\times[0,1]$ for every $i=1,\ldots,n-2$.
The right end is forced by reflection.

Let $B(X;u,v)$ be the coefficient obtained from these forced plaquettes
and the four selected half-paths. We have proved
\[
 A_\beta(X;u,v)
 =\beta^{4n-4}\bigl(B(X;u,v)+O_{n,L}(\beta)\bigr).
\]
The expansion is uniform in $X$ and $u,v\in E$: the retained discrete
indices form a finite set, and the finite-volume Taylor expansion is
uniform on the compact space of remaining boundary variables.
Gluing four copies of this configuration restores the unit frame tube
$\mc{T}$ bounding
\[
 \bigl([0,n]^2\setminus(1,n-1)^2\bigr)\times[0,1].
\]
It has $F=4|\ell|-16$ faces. Summing the shared indices restores each of
the eight selected plaquette traces; every other tube plaquette occurs
once. All corresponding factorials are one. Hence
\begin{equs}\label{eq:truncated-leading-term}
 \widetilde{\mc{N}}(\beta)
 =\beta^F\bigl(I_\ell(\mathbf1_{\mc{T}})+O_{n,L}(\beta)\bigr).
\end{equs}

\emph{The tube integral.}
Write $v_{\mc{T}}$ and $e_{\mc{T}}$ for the numbers of vertices and
edges of the tube surface. Each edge belongs to two faces, so
$e_{\mc{T}}=2F$. Since the surface is a torus,
$v_{\mc{T}}-e_{\mc{T}}+F=0$, and hence $v_{\mc{T}}=F$.

Expand all traces into matrix entries and apply
Lemma~\ref{lm:Haar-SO(3)-Moments}. Each edge outside $\ell$ occurs twice.
Its second moment contributes $1/3$ and identifies the indices of its
two incident faces at each endpoint. At a vertex outside $\ell$, these
identifications equate all incident face indices, leaving one index
with three possible values.

\todo{this paragraph is confusing!} At a vertex on $\ell$, the incident faces instead form two chains, one
on each side of the oriented loop. The second-moment identifications
leave one index for each chain, together with the Wilson index. Order
these three indices as Wilson, left, right, consistently around $\ell$.
By the third-moment formula, the triple at each vertex must be a
permutation of $(1,2,3)$, giving six choices. For each oriented loop
edge $e$, let $\sigma_e\in S_3$ carry the ordered triple at its initial
vertex to that at its terminal vertex. Its Haar moment is
$\mrm{sgn}(\sigma_e)/6$. The composition of these permutations around
$\ell$ is the identity, so their signs multiply to $1$.

Thus every surviving index assignment has positive sign, there are
$3^{v_{\mc{T}}-|\ell|}6^{|\ell|}$ such assignments, and
\begin{equs}\label{eq:tube-integral}
 I_\ell(\mathbf1_{\mc{T}})
 =\frac13\,3^{v_{\mc{T}}-|\ell|}6^{|\ell|}
              3^{-(e_{\mc{T}}-|\ell|)}6^{-|\ell|}
 =3^{-F-1}>0.
\end{equs}
Together with \eqref{eq:numerator-comparison} and
\eqref{eq:truncated-leading-term}, this proves
\eqref{eq:square-numerator-lower}.
\end{proof}

\begin{figure}
    \centering
    \includegraphics[width=.65\textwidth]{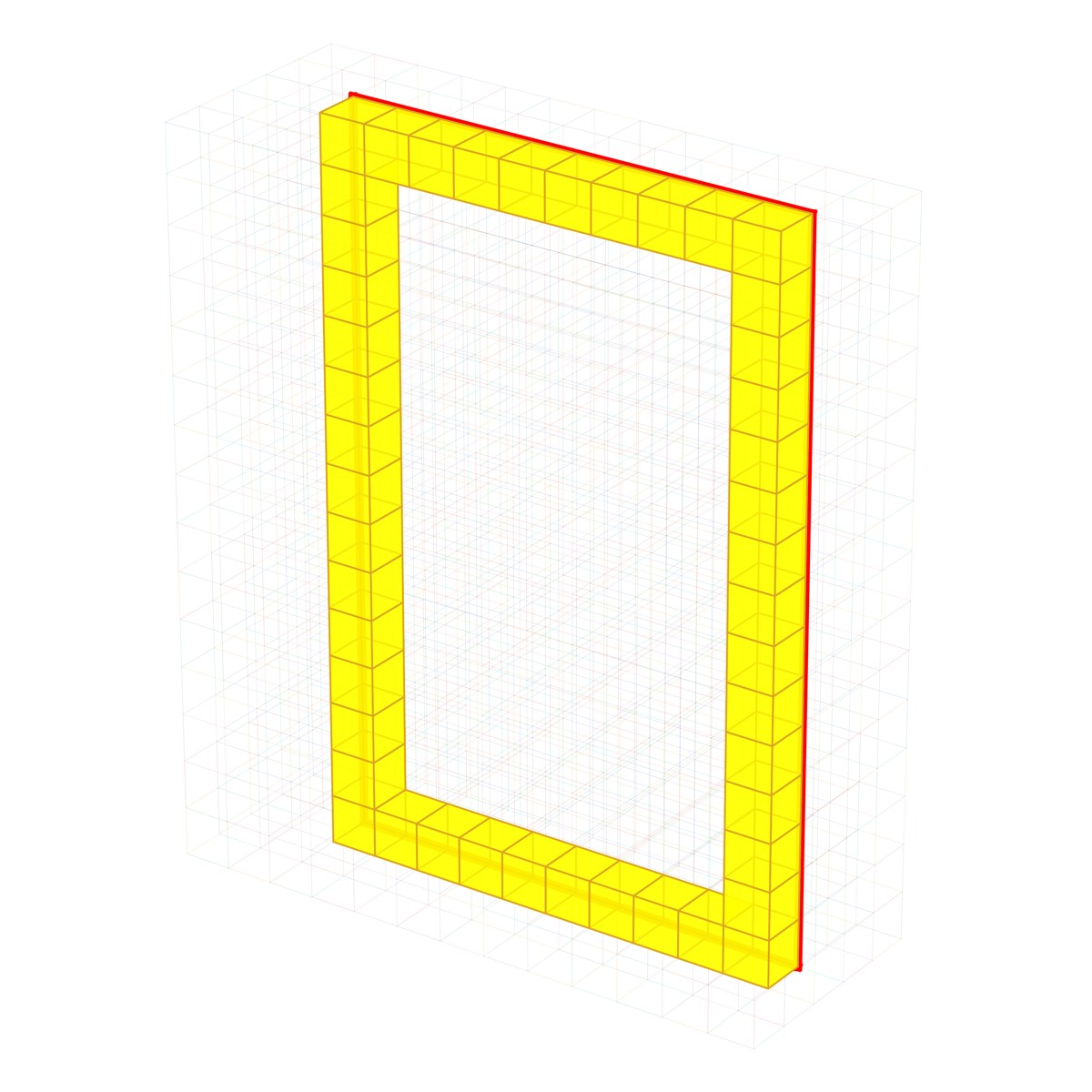}
    \caption{The yellow plaquettes form a unit frame tube $\mc{T}$ containing the red Wilson loop $\ell$ on its surface. The tube has $4|\ell|-16$ faces. For an even square, the diagonal cuts divide it into four identical parts.}
    \label{fig:tube-diagram}
\end{figure}

\begin{proof}[Proof of Proposition~\ref{lm:Tube-Lemma}]

\emph{The first degree.}
Write $w_{R,T}(\beta):=
\langle W_{\ell_{R,T}}\rangle_{\Lambda_\infty,\beta}$.
Fixed Taylor coefficients of the finite-volume expectation stabilize
as $L\to\infty$: the coefficient of order $j$ is a sum of product-Haar
cumulants rooted at $\ell$, and every nonzero term uses only plaquettes
connected to $\ell$ by at most $j$ edge-sharing plaquettes. The analytic
convergence in Lemma~\ref{lm:cluster-expansion}(1) identifies these with
the coefficients of $w_{R,T}$. Since $Z_{\Lambda_L,0}=1$, multiplication
by $Z_{\Lambda_L,\beta}$ does not change the first nonzero degree or
coefficient. For a fixed even square, choose a symmetric box large
enough for stabilization through degree $16n-16$.
Lemma~\ref{lm:square-noncancellation} then gives
\[
 m(n,n)\le16n-16\qquad(n\ge4\text{ even}),
\]
where $m(R,T)$ denotes the first nonzero degree of $w_{R,T}$ whenever
it exists.

By Lemmas~\ref{lm:Wilson-nonnegative}--\ref{lm:v-is-concave},
$w_{R,T}(\beta)\ge w_{n,n}(\beta)\ge0$ for
$n\ge\max\{R,T\}$ and sufficiently small positive $\beta$.
Thus every $m(R,T)$ is finite, and we may write
\[
 w_{R,T}(\beta)
 =a(R,T)\beta^{m(R,T)}\bigl(1+O_{R,T}(\beta)\bigr),
 \qquad a(R,T)>0.
\]
Lemma~\ref{lm:slab-bounds} also gives
$m(R,T)\ge8\max\{R,T\}$ when $R,T\ge8$.

Substituting these expansions into the reflection inequality
\[
 w_{R,T}(\beta)^2\le
 w_{R-1,T}(\beta)w_{R+1,T}(\beta)
\]
shows that
\[
 2m(R,T)\ge m(R-1,T)+m(R+1,T).
\]
Indeed, the opposite strict inequality would make the left-hand side
have smaller order than the right-hand side, contradicting the
inequality for small positive $\beta$. The same argument applies in
$T$. Hence $m$ is discretely concave in both variables.

Fix $T\ge8$. The increments
$m(R+1,T)-m(R,T)$ are nonincreasing in $R$. None can be less than $8$
for $R\ge8$: if one equals $\delta<8$ at $R=R_0$, then
\[
 m(R_0+k,T)\le m(R_0,T)+k\delta\qquad(k\ge1),
\]
contradicting $m(R_0+k,T)\ge8(R_0+k)$ for large $k$.
The same reasoning applies to the increments in $T$.
Consequently, for even $n\ge\max\{R,T\}$ with $R,T\ge8$,
\[
 m(R,T)+8(n-R)+8(n-T)\le m(n,n)\le16n-16.
\]
Rearranging gives $m_\ell\le4|\ell|-16$.
For $R,T\ge100$, \eqref{eq:slab-support-bound} gives the lower bound
in \eqref{eq:leading-coefficient}.

\emph{The first coefficient.}
The square lower bound controls the first degree, not its coefficient
when earlier terms occur. We therefore use a common denominator.
Take $L$ large enough for stabilization through degree $4|\ell|-16$.
Every nonzero summand at degree $m=m_\ell$ satisfies
\[
 |K|-|\mrm{supp}K|\le112,
 \qquad K(p)\le113.
\]
Each edge matrix consequently occurs at most
$q=1+226(d-1)$ times. Choose an integer $D_d$ divisible by $3$, by
$113!$, and by the denominators of all Haar moments of at most $q$
matrix entries. These moments are rational: rotation-matrix entries
are integer-coefficient quadratic polynomials in a uniform unit
quaternion, whose even moments, for $\sum b_i=r$, are
\[
 \int_{S^3}\prod_{i=0}^3x_i^{2b_i}\,d\sigma(x)
 =\frac{\prod_{i=0}^3(2b_i-1)!!}{2^r(r+1)!};
\]
monomials odd in any coordinate have integral zero.

There are at most $4m+|\ell|$ active edge variables and at most $m$
supported plaquettes. Expanding traces into matrix entries therefore
gives the common denominator
\[
 \frac{I_\ell(K)}{K!}\in
 D_d^{-(5m+|\ell|+1)}\Z
 \subseteq D_d^{-21|\ell|}\Z.
\]
The sum $a_\ell$ belongs to the same set. Since it is positive,
$a_\ell\ge D_d^{-21|\ell|}$, proving the proposition with
$C_d=D_d^{21}$.
\end{proof}

Now we recall the standard cluster-expansion estimates needed below. Unlike the leading-coefficient estimate, these do not require any vanishing statement for small clusters.
\begin{lemma}[Cluster expansion corollaries]\label{lm:cluster-expansion}
There are $r_0=r_0(d)\in(0,1]$ and $A_d\ge1$ such that:
\begin{enumerate}
    \item For every fixed loop $\ell$, the finite-volume expectations are holomorphic on $\mbb{D}(0,2r_0)$ and converge locally uniformly there to a holomorphic function $\langle W_\ell\rangle_{\Lambda_\infty,\beta}$. For real $\beta$ in this disk, this is the expectation in the unique infinite-volume limit.
    \item $Z_{\Lambda_L,\beta}\ne0$ on $\mbb{D}(0,2r_0)$ for every finite $L$.
    \item For every nonempty loop $\ell$ and every box containing it,
    \begin{equs}\label{eq:Wilson-upper-bound}
        \sup_{|\beta|\le r_0}|\langle W_\ell\rangle_{\Lambda_L,\beta}|
        \le A_d^{|\ell|}.
    \end{equs}
    The same bound holds in infinite volume.
\end{enumerate}
\end{lemma}
\begin{proof}
The first two statements follow from the convergent cluster expansion; see \cite[Theorems 3.6--3.7 and Lemma 3.2]{osterwalderseiler1978}. We give the rooted expansion that yields the third statement, making the deletion convention explicit.

Let $\mc{P}=\mc{P}_{\Lambda_L}^+$, and, for $\mc{A}\subseteq\mc{P}$, let
\[
 Z_{\mc{A}}(\beta):=\int\prod_{p\in\mc{A}}e^{\beta\mrm{Tr}(Q_p)}\,dQ.
\]
For $\mc{C}\subseteq\mc{P}$, write $\partial\mc{C}$ for the union of its unoriented boundary edges. Identifying $\ell$ with its unoriented edge set in the following definition, put
\[
 \mc{D}_\ell(\mc{C})
 :=\{p\in\mc{P}:\partial p\cap(\ell\cup\partial\mc{C})\ne\varnothing\},
 \qquad
 \widehat Z_{\ell,\mc{C}}(\beta):=Z_{\mc{P}\setminus\mc{D}_\ell(\mc{C})}(\beta).
\]
Thus the remaining plaquettes share no edge variable with the loop or with the rooted cluster. Expanding the activities and collecting the component attached to $\ell$ gives
\begin{equs}\label{eq:rooted-cluster-expansion}
 \langle W_\ell\rangle_{\Lambda_L,\beta}
 =\sum_{\substack{\mc{C}\subseteq\mc{P}\\\mc{C}\cup\ell\ \mathrm{connected}}}
 \int W_\ell(Q)\prod_{p\in\mc{C}}
       (e^{\beta\mrm{Tr}(Q_p)}-1)\,dQ\,
 \frac{\widehat Z_{\ell,\mc{C}}(\beta)}{Z_{\Lambda_L,\beta}}.
\end{equs}
Here connectivity is by shared edge variables, with the loop treated as one root: every component of $\mc{C}$ must be joined to a loop edge.

The absolute convergence of the cluster expansion bounds the change in $\log Z$ upon deleting $s$ plaquette interactions by a constant times $s$, uniformly on a sufficiently small closed disk. Since $|\mc{D}_\ell(\mc{C})|\le c_d(|\ell|+|\mc{C}|)$, it follows that
\[
 \left|\frac{\widehat Z_{\ell,\mc{C}}(\beta)}{Z_{\Lambda_L,\beta}}\right|
 \le B_d^{|\ell|+|\mc{C}|}.
\]
The number of rooted clusters of size $n$ is at most $B_d^{|\ell|+n}$, after increasing $B_d$; this is the usual connected-set count in the bounded-degree plaquette adjacency graph. Finally,
\[
 |e^{\beta\mrm{Tr}(Q_p)}-1|\le e^{3|\beta|}-1.
\]
Choose $r_0$ small enough that $B_d^2(e^{3r_0}-1)<1/2$. The resulting geometric series in $|\mc{C}|$ proves \eqref{eq:Wilson-upper-bound}, with a suitably enlarged $A_d$. This estimate uses the full activity rather than a bound on an individual plaquette multiplicity. Passing to the locally uniform limit gives the infinite-volume bound.
\end{proof}

Proposition \ref{lm:Tube-Lemma} and Lemma \ref{lm:cluster-expansion} now supply the two inputs to the complex analysis argument.
\begin{lemma}\label{lm:Input-for-complex-analysis-lemma}
Suppose $\ell$ is a rectangular loop with both side lengths at least $100$. With $m_\ell,a_\ell$ from Proposition \ref{lm:Tube-Lemma}, the function
\begin{equs}\label{eq:regularized-Wilson-loop}
 \phi_\ell(\beta):=\beta^{-m_\ell}\langle W_\ell\rangle_{\Lambda_\infty,\beta}
\end{equs}
extends holomorphically to $\mbb{D}(0,2r_0)$, and
\begin{equs}\label{eq:complex-analysis-input}
 \phi_\ell(0)=a_\ell\ge C_d^{-|\ell|},
 \qquad
 \sup_{|\beta|\le r_0}|\phi_\ell(\beta)|\le H_d^{|\ell|},
 \qquad H_d:=A_dr_0^{-4}.
\end{equs}
\end{lemma}
\begin{proof}
The first degree and coefficient of the normalized expectation are those of the numerator, since $Z_{\Lambda_L,0}=1$. The proposition therefore gives the holomorphic extension and its value at zero. On $|\beta|=r_0$, Lemma \ref{lm:cluster-expansion}(3) and $m_\ell\le4|\ell|$ give
\[
 |\phi_\ell(\beta)|\le r_0^{-m_\ell}A_d^{|\ell|}
 \le (A_dr_0^{-4})^{|\ell|}.
\]
The maximum modulus principle gives the same bound on the disk.
\end{proof}

\subsection{Step 2: A key insight from complex analysis}\label{subsection:complex-analysis}\label{Section:Step-2}

In this step, the goal is to use the cluster analysis of the previous section to derive a perimeter law lower bound for any fixed large rectangular loop outside a set of inverse couplings of arbitrarily small prescribed measure. The key input we need besides the cluster expansion and leading-coefficient estimate from the previous section is the following complex analysis lemma. The proof is elementary, relying on standard results of complex analysis found in \cite{ahlfors1979complex}, as well as one additional lemma known as Cartan's lemma \cite[Theorem 9]{levin1980distribution} which controls the set where a polynomial is unusually small.

\begin{lemma}\label{lm:Complex-Analysis-Lemma}
    Suppose $f$ is holomorphic on $\mbb{D}(0,r)$, with $|f(0)|\ge m>0$ and $|f(z)|\le M<\infty$ throughout the disk. Then for any $\epsilon \in (0,r)$ there exists a collection of disks $\{D_i\}$ of radii $r_i$ with $\sum_i r_i\le\epsilon$ such that
    \begin{equs}
        |f(z)|\geq m (m/M)^{\frac{\log(\frac{3er}{2\epsilon})}{\log(2)}+2}
    \end{equs}
     on $\mathbb{D}(0,r/4) \backslash(\bigcup_{i} D_i)$.
\end{lemma}

\begin{proof}
    First, a standard application of Jensen's formula shows that the number of zeros (counted with multiplicity) of $f$ in the disk of half the radius $\mathbb{D}(0,r/2)$, $N_{r/2}(f)$, is bounded above as follows,
    \begin{equs}
        N_{r/2}(f) \leq \frac{\log(M/m)}{\log(2)}.
    \end{equs}
    Define $\tilde f(z):=f(rz/2)$. It is holomorphic on $\mbb{D}(0,2)$, and its zeros in the unit disk are $z_1,\dots,z_N$, counted with multiplicity, where $N\le\log(M/m)/\log2$.

    From the considerations of the last paragraph, we can write the Blaschke product factorization 
    \begin{equs}
        \tilde{f}(z)= e^{g(z)} \prod_{j=1}^{N} \frac{z-z_j}{\overline{z_j}z-1},
    \end{equs}
    where $g$ is holomorphic on the unit disk. The quotient by the finite Blaschke product is holomorphic across the closed unit disk, and each Blaschke factor has modulus one on its boundary. The maximum modulus principle therefore gives $\sup_{z\in\mbb{D}(0,1)}|e^{g(z)}|\le M$.

    Next we observe that $u(z):=\log M -\mrm{Re}\, g(z)$ is a nonnegative harmonic function on $\mbb{D}(0,1)$, and plugging in $z=0$ into the factorization and recalling that $|z_j|<1$ for each $j$, we have $|e^{g(0)}| \geq m$, and so $u(0) \leq \log(M/m)$. Thus by Harnack's inequality,
    \begin{equs}
        |u(z)| \leq 3 u(0) \leq 3\log(M/m),
    \end{equs}
    uniformly on $\mathbb{D}(0,\frac{1}{2})$. As a result,
    \begin{equs}
        |e^{g(z)}| \geq \frac{m^3}{M^2}. 
    \end{equs}

    Next we control the set where $B(z):=\prod_{j=1}^{N} \frac{z-z_j}{\overline{z_j}z-1}$ is unusually small. For $|z|\leq \frac{1}{2}$, $|\overline{z_j}z-1| \leq \frac{3}{2}$, so
    \begin{equs}
        \prod_{j=1}^{N} \frac{1}{|\overline{z_j}z-1|}\geq (2/3)^N.
    \end{equs}
     If $N=0$, take an empty family of exceptional disks. Otherwise, by Cartan's lemma \cite[Theorem 9]{levin1980distribution}, for any $H \in (0,1)$, there exists a collection of disks $\{D_i\}$ with radii $r_i$ satisfying $\sum_{i} r_i \leq 2H$ such that,
    \begin{equs}
        \prod_{j=1}^{N}|z-z_j| \geq (H/e)^N
    \end{equs}
    outside of these disks. As a result
    \begin{equs}
        |B(z)| &\geq \bigg(\frac{2H}{3e}\bigg)^N
    \end{equs}

    Putting everything together and plugging in $N \leq \frac{\log(M/m)}{\log(2)}$
    \begin{equs}
        |\tilde{f}(z)| \geq m (m/M)^{\frac{\log(\frac{3e}{2H})}{\log(2)}+2}
    \end{equs}
    on $\mbb{D}(0,\frac{1}{2}) \backslash (\cup_i D_i)$. Rescaling the domain by a factor $r/2$ and setting $\epsilon=rH$ completes the proof.
\end{proof}

\begin{lemma}[Single rectangle perimeter law lower bound]\label{lm:single-rectangle-bound}
Put $\beta^*(d)=r_0/4$. For every $\epsilon\in(0,\beta^*(d))$, there is a function $c=c(\beta,\epsilon,d)\in(0,1)$, independent of the loop, such that for every rectangular loop $\ell$ with both side lengths at least $100$ there is a measurable set $A_{\epsilon,\ell}\subset(0,\beta^*(d))$ satisfying
\begin{equs}\label{eq:single-rectangle-bound}
 |A_{\epsilon,\ell}|>\beta^*(d)-\epsilon,
 \qquad
 |\langle W_\ell\rangle_{\Lambda_\infty,\beta}|
 \ge c(\beta,\epsilon,d)^{|\ell|}
 \quad(\beta\in A_{\epsilon,\ell}).
\end{equs}
\end{lemma}
\begin{proof}
Apply Lemma \ref{lm:Complex-Analysis-Lemma} to $f=\phi_\ell$, using
\[
 m=C_d^{-|\ell|},\qquad M=H_d^{|\ell|},\qquad r=r_0,
\]
and exceptional-radius budget $\epsilon/4$. The exceptional disks intersect the real interval in a set of length at most the sum of their diameters, hence at most $\epsilon/2$. Outside this set,
\[
 |\phi_\ell(\beta)|
 \ge\left[C_d^{-1}(C_d^{-1}/H_d)^{\kappa_\epsilon}\right]^{|\ell|},
 \qquad
 \kappa_\epsilon:=\frac{\log(6er_0/\epsilon)}{\log 2}+2.
\]
Since $0<\beta<r_0/4<1$ and $m_\ell\le4|\ell|$, restoring the factor $\beta^{m_\ell}$ gives the assertion with
\[
 c(\beta,\epsilon,d)
 :=\beta^4C_d^{-1}(C_d^{-1}/H_d)^{\kappa_\epsilon}.
\]
\end{proof}

\subsection{Step 3: Some corollaries of reflection positivity}\label{subsection:Ref-Pos}

In this final step, we use reflection positivity considerations to go from the single loop perimeter law bound Lemma \ref{lm:single-rectangle-bound}, to the uniform bound on the quark-antiquark potential Theorem \ref{thm:Main-theorem}.

We review the reflection positivity property of lattice Yang--Mills theory, proved in \cite{osterwalderseiler1978,glimm2012quantum}. We state the forms needed for $G=\mrm{SO}(3)$ and work in the infinite-volume strong coupling limit. The arguments begin in finite reflection-symmetric boxes with free boundary conditions. Throughout this subsection, $\langle\cdot\rangle$ denotes $\langle\cdot\rangle_{\Lambda_\infty,\beta}$, with $\beta>0$ sufficiently small. None of the reflection arguments uses Proposition \ref{lm:Tube-Lemma}.

\begin{definition}[Coordinate reflections]
For $P=\{x\in\R^d:x_i=a\}$, where $a\in\frac12\Z$, let $\theta_P$ be geometric reflection across $P$, and define
\[
 \Theta_P(x,y):=(\theta_Px,\theta_Py),
 \qquad
 (\Theta_PQ)_e:=Q_{\Theta_Pe}.
\]
Thus the positive coordinate orientation is reversed for edges perpendicular to $P$, not for those parallel to it. For a local observable $f$, set
\begin{equs}
 (\Theta_Pf)(Q):=f(\Theta_PQ).
\end{equs}
\end{definition}

\begin{definition}[Matrix-valued local observables]
A matrix-valued local observable is a function $F:G^{E_{\Lambda_\infty}^+}\to\C^{n\times n}$ whose entries are local observables. Dependence on edges and the action of $\Theta_P$ are understood entrywise.
\end{definition}

\begin{lemma}[Reflection positivity]\label{lm:reflec-pos}
For an integer reflection plane $P$ and matrix-valued local observables $F,G$ depending only on edges in the same closed half-lattice, the form
\[
 \mc{B}_P(F,G):=\langle\mrm{Tr}(F(\Theta_PG)^*)\rangle
\]
is Hermitian and satisfies
\begin{equs}\label{eq:Reflection-Positivity}
 \mc{B}_P(F,F)\geq 0,
\end{equs}
and
\begin{equs}\label{eq:Cauchy-Schwarz}
 |\mc{B}_P(F,G)|^2\le\mc{B}_P(F,F)\mc{B}_P(G,G).
\end{equs}
For a half-integer plane $P$, let $F,G$ be matrix-valued local observables depending only on edges in the same half-lattice. Suppose that, for $H,H'\in\{F,G\}$, there are continuous gauge invariant scalar local observables $A_{H,H'}$ satisfying
\begin{equs}\label{eq:bond-reflection-gluing}
 A_{H,H'}(Q)=\mrm{Tr}\bigl(H(Q)(\Theta_PH')(Q)^*\bigr)
\end{equs}
whenever every edge crossing $P$ has variable $I$. Then the same Hermitian symmetry and conclusions \eqref{eq:Reflection-Positivity} and \eqref{eq:Cauchy-Schwarz} hold for
\[
 \mc{B}_P(H,H'):=\langle A_{H,H'}\rangle,
 \qquad H,H'\in\{F,G\}.
\]
\end{lemma}
\begin{proof}
First replace $\Z^d$ with a finite box which is exactly bisected by $P$ and contains the supports and reflected supports of the observables. For an integer plane, condition on the edge variables contained in $P$. The two half-lattice configurations are then conditionally independent, and their conditional laws are interchanged by reflection. For scalar $F,G$,
\[
 \E[F\,\overline{\Theta_PG}\mid(Q_e)_{e\subset P}]
 =\E[F\mid(Q_e)_{e\subset P}]
  \,\overline{\E[G\mid(Q_e)_{e\subset P}]}.
\]
Integrating and applying this identity entrywise gives a positive semidefinite Hermitian form.

For a half-integer plane, assume $F,G$ are supported in the left half-lattice and orient every crossing edge $e=(x,y)$ from left to right. These edges are pairwise vertex-disjoint. For fixed crossing variables $q_e$, choose the gauge transformation with $\phi_y=q_e$ at their right endpoints and $\phi_z=I$ at every other vertex. This sets all crossing variables to the identity and preserves product Haar measure on the remaining edges. Gauge invariance of the action and each $A_{H,H'}$ therefore allows the crossing variables to be integrated out, with Haar integral one. The same argument with observable $1$ preserves the partition function. The gluing identity now replaces $A_{H,H'}$ by $\mrm{Tr}(H(\Theta_PH')^*)$.

A plaquette crossing $P$ has gauge-fixed weight
\[
 e^{\beta\mrm{Tr}(UV^T)}
 =\sum_{k\ge0}\frac{\beta^k}{k!}
   \sum_{I,J\in\{1,2,3\}^k}
   (U^{\otimes k})_{IJ}(V^{\otimes k})_{IJ},
\]
where $U,V$ are the corresponding edge matrices on the two halves. Index the choices of $k,I,J$ over all crossing plaquettes by $\alpha$. Let $c_\alpha\geq0$ be the product of the coefficients $\beta^k/k!$, and let $I_{H,\alpha}$ be the integral over the left half of $H$ times its internal plaquette weights and the selected matrix coefficients. Reflection and Fubini give
\[
 Z_{\Lambda,\beta}\langle A_{H,H'}\rangle_{\Lambda,\beta}
 =\sum_\alpha c_\alpha\,
   \mrm{Tr}\bigl(I_{H,\alpha}I_{H',\alpha}^*\bigr).
\]
The expansion converges absolutely and uniformly before integration, since every matrix entry is bounded by one. Thus the four pairings form a positive semidefinite Hermitian matrix. Cauchy--Schwarz follows in both cases.

Finally, exhaust $\Z^d$ by reflection-symmetric boxes and pass to the unique translation invariant infinite-volume limit. In the half-integer case, extend gauge-fixed configurations by the identity outside each box to check the gluing hypothesis, and use continuity of the local observables $A_{H,H'}$ to pass their expectations to the limit. For integer planes, first pass to the limit for continuous local observables; bounded Borel observables then follow by $L^2$ approximation with continuous observables on the same finite edge set, using reflection invariance of the limit.
\end{proof}

\begin{remark}
For a half-integer plane, let $a_-,b_-$ be boundary vertices of the left half-lattice, write $a_+=\theta_Pa_-$ and $b_+=\theta_Pb_-$, and set
\[
 Q_a:=Q_{(a_-,a_+)},\qquad Q_b:=Q_{(b_-,b_+)}.
\]
For path matrices $F,G$ from $a_-$ to $b_-$ in the left half-lattice, or linear combinations with these endpoints, the gauge invariant gluing gives
\begin{equs}\label{eq:bond-reflection-form}
 \mc{B}_P(F,G)
 :=\langle\mrm{Tr}(F Q_b(\Theta_PG)^*Q_a^*)\rangle.
\end{equs}
The crossing matrices close the two paths into a loop and become identities after gauge fixing. More generally, the gauge invariance requirement is on the glued observables $A_{H,H'}$, not on $F,G$ themselves. The integer-plane statement is the usual entrywise form of reflection positivity; see \cite[Theorem 2.1]{osterwalderseiler1978} and \cite[Theorem 22.3.1]{glimm2012quantum}.
\end{remark}

Let $\ell_{R,T}$ be a coordinate rectangle of side lengths $R,T$. Its position is immaterial by translation invariance. Write
\[
 w(R,T):=\langle W_{\ell_{R,T}}\rangle,
 \qquad w(R,0)=w(0,T):=1,
\]
where a rectangle of zero width is interpreted as a backtracking loop, and define
\begin{equs}\label{eq:log-Wilson-loop}
 v(R,T):=-\log w(R,T),
\end{equs}
initially allowing the value $+\infty$.

\begin{lemma}[Rectangular Wilson loop expectations are nonnegative]\label{lm:Wilson-nonnegative}
For all nonnegative integers $R,T$, $w(R,T)\in[0,1]$.
\end{lemma}
\begin{proof}
Bisect the rectangle by a plane perpendicular to its length-$R$ sides. If $R$ is even, its holonomy is the product of a half-path matrix and the adjoint of its reflection, so positivity follows from the integer-plane form of Lemma \ref{lm:reflec-pos}. If $R$ is odd, retain the bottom and top crossing edges and use \eqref{eq:bond-reflection-form}. In either case the corresponding reflection form equals $3w(R,T)$. The upper bound follows from $|W_\ell|\le1$.
\end{proof}

The next properties are standard \cite{QuarkPotentialConcave}. We include the argument, keeping track of the two possible parities of a cut.
\begin{lemma}[Properties of the quark-antiquark potential]\label{lm:v-is-concave}
After decreasing the strong coupling interval if necessary, $w(R,T)>0$ for every $R,T\ge1$. The finite function $v(R,T)$ satisfies:
\begin{enumerate}
    \item It is concave in each variable. For nonnegative integers $r\le R$ and $t\le T$,
    \begin{equs}\label{eq:v-is-concave}
      v(R,T)&\ge\tfrac12\bigl(v(R-r,T)+v(R+r,T)\bigr),\\
      v(R,T)&\ge\tfrac12\bigl(v(R,T-t)+v(R,T+t)\bigr).
    \end{equs}
    \item It is subadditive in $T$, and
    \begin{equs}\label{eq:potential-infimum}
      V(R):=\lim_{T\to\infty}\frac{v(R,T)}T
      =\inf_{T\ge1}\frac{v(R,T)}T<\infty.
    \end{equs}
    Moreover, $V(R)$ is concave.
    \item $v(R,T)$ is nondecreasing in both variables, and $V(R)$ is nondecreasing.
\end{enumerate}
\end{lemma}
\begin{figure}
    \centering
    \includegraphics[width=\textwidth]{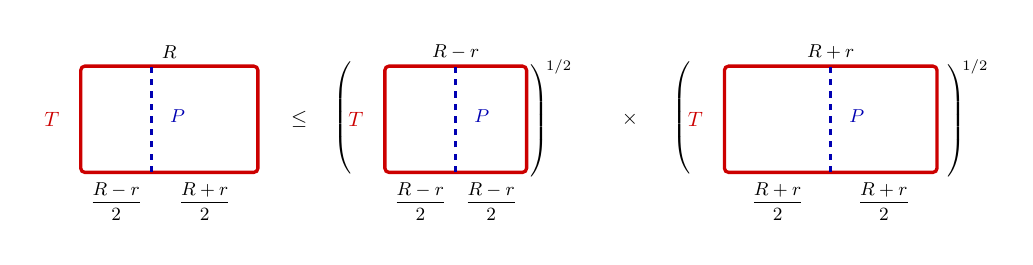}
    \caption{Reflection positivity compares rectangles through
    $w(R,T)^2\le w(R-r,T)w(R+r,T)$.
    The reflected loops have widths $R-r$ and $R+r$. For a half-integer cut, the two crossing edges are included as in \eqref{eq:bond-reflection-form}.}
    \label{fig:v-concave}
\end{figure}
\begin{proof}
\textit{Proof of (1) and strict positivity.}
We first work with $w$ before taking logarithms. For left-half rectangular paths of widths $a,b$ with common endpoints separated by $T$, a site reflection gives
\[
 w(a+b,T)^2\le w(2a,T)w(2b,T).
\]
A bond reflection, with the crossing edges included, gives
\[
 w(a+b+1,T)^2\le w(2a+1,T)w(2b+1,T).
\]
If $R-r$ is even, take $a=(R-r)/2$, $b=(R+r)/2$ in the first inequality. If it is odd, take $a=(R-r-1)/2$, $b=(R+r-1)/2$ in the second. Thus
\begin{equs}\label{eq:Wilson-log-convexity}
 w(R,T)^2\le w(R-r,T)w(R+r,T)
\end{equs}
for both parities. Exchanging the coordinates gives the analogous inequality in $T$.

For a single plaquette, cluster expansion gives
\[
 w(1,1;\beta)=\frac{\beta}{3}+O(\beta^2),
\]
where the big $O$ dependence is only in the dimension $d$ and is uniform in volume.
Indeed, at degree one only the same plaquette contributes, and Lemma \ref{lm:Haar-SO(3)-Moments} gives $\E_0[(\mrm{Tr}Q_p)^2]=1$. Thus in a sufficiently small strong coupling threshold, $w(1,1;\beta)>0$ on the working interval. Applying \eqref{eq:Wilson-log-convexity} with increment one first in $T$ and then in $R$ propagates positivity to every rectangle. The zero-length values are one. We may therefore take logarithms in the two reflection inequalities, proving (1).

\textit{Proof of (2).}
For $f(T)=v(R,T)$ we have $f(0)=0$, and concavity says that its successive increments are nonincreasing. Consequently,
\[
 f(n+m)-f(n)
 =\sum_{j=1}^m\bigl(f(n+j)-f(n+j-1)\bigr)
 \le\sum_{j=1}^m\bigl(f(j)-f(j-1)\bigr)=f(m).
\]
Thus $f$ is subadditive. Fekete's lemma gives \eqref{eq:potential-infimum}; finiteness follows from $0\le V(R)\le v(R,1)<\infty$. Divide the first inequality in \eqref{eq:v-is-concave} by $T$ and pass to the limit to obtain concavity of $V$.

\textit{Proof of (3).}
A finite concave sequence bounded below is nondecreasing: a negative increment would force every later increment to be at least as negative, contradicting the lower bound. Apply this to each variable of $v$, and then either use \eqref{eq:potential-infimum} or pass to the limit to obtain monotonicity of $V$.
\end{proof}

We are now ready to upgrade the single-loop estimate to Theorem \ref{thm:Main-theorem}. It is enough to use large even squares.
\begin{proof}[Proof of Theorem \ref{thm:Main-theorem}]
Fix $\epsilon\in(0,\beta^*)$. For each even $n\ge100$, let $A_{\epsilon,n}$ be the set supplied by Lemma \ref{lm:single-rectangle-bound} for $\ell_{n,n}$, and define
\[
 A_\epsilon
 :=\bigcap_{M\ge100}\ \bigcup_{\substack{n\ge M\\n\ \mathrm{even}}}A_{\epsilon,n}.
\]
Every tail union has measure greater than $\beta^*-\epsilon$, so continuity of measure gives $|A_\epsilon|\ge\beta^*-\epsilon$.

Fix $\beta\in A_\epsilon$, and write $q(\beta,\epsilon,d):=-\log c(\beta,\epsilon,d)$. There are arbitrarily large even $n$ for which
\[
 v(n,n)\le4nq(\beta,\epsilon,d).
\]
For any $R\ge1$, choose such an $n\ge R$. Lemma \ref{lm:v-is-concave} gives
\[
 V(R)\le\frac{v(R,n)}n\le\frac{v(n,n)}n
 \le4q(\beta,\epsilon,d).
\]
The right side is finite and independent of $R$. Finally, choose $\epsilon_j\downarrow0$. The union $\bigcup_j A_{\epsilon_j}$ has full measure in $(0,\beta^*)$, and $\sup_RV(R)<\infty$ on this union, as required.
\end{proof}

\section{Comments about excitations to the tube cluster}\label{section:comments-on-cluster-expansion}

We briefly explain why a leading tube computation alone does not give a lower bound uniform in the loop size. Use the rooted cluster expansion \eqref{eq:rooted-cluster-expansion}, with the residual partition function defined there. A tube $\mc{T}$ as in Figure \ref{fig:tube-diagram} has $F=4|\ell|-16$ plaquettes and contributes a term whose first Taylor coefficient is
\[
 I_\ell(\mathbf1_{\mc{T}})=3^{-F-1}.
\]
The tube computation in the proof of Lemma \ref{lm:square-noncancellation} applies to this rectangular tube as well. Proposition \ref{lm:Tube-Lemma} does not assert that the tubes are the only configurations at the first nonzero degree.

Choose a plaquette of $\mc{T}$ away from the corners and disjoint from $\ell$, together with an outward unit cube whose intersection with the tube is precisely that face. Remove the chosen face and replace it by the other five faces of the cube. The resulting cluster $\mc{T}'$ has $F+4$ plaquettes. Integrating the new edges using Lemma \ref{lm:double-edge-integration} gives
\[
 I_\ell(\mathbf1_{\mc{T}'})=3^{-4}I_\ell(\mathbf1_{\mc{T}})>0.
\]
There are at least $c|\ell|$ and at most $C|\ell|$ such single-cube perturbations, for absolute positive constants $c,C$. For instance, one may use the faces in the upper planar annulus away from its corners. Since all residual partition-function ratios equal one at $\beta=0$, the sum of their individual leading Taylor terms, divided by the leading term for one tube, is
\[
 \frac{\#\{\text{chosen single-cube perturbations}\}}{81}\,\beta^4,
\]
which is of order $|\ell|\beta^4$.

This comparison is at the level of the selected Taylor terms. The remainders for the full cluster activities need not be uniform in $\ell$, and other configurations may have negative integrals. Thus it does not prove a fixed-coupling asymptotic as $|\ell|\to\infty$. It does show why a fixed-loop leading-order expansion does not establish the uniform tube dominance suggested by the arguments in \cite{GlimmJaffeNonConfinement,seiler1982gauge}. The sum-of-squares comparison in Step 1 and the complex analysis argument in Step 2 avoid the need for such a dominance estimate.

\section{Comments about larger orthogonal groups and the large $N$ limit}\label{Section:large-N}

One expects the arguments of this paper to extend to $\mrm{SO}(2N+1)$ for $N>1$, but the single-copy tube used above contributes zero in that case: every Wilson edge has three matrix-entry occurrences. A nonzero analogue must use higher multiplicities or a different configuration. The relevant elementary observation is the following, where $M$ denotes the matrix dimension.
\begin{lemma}\label{lm:odd-orthogonal-moments}
If $k$ is odd and $k<M$, and $Q\sim\mrm{Haar}(\mrm{SO}(M))$, then
\begin{equs}
 \E[Q^{\otimes k}]=0.
\end{equs}
\end{lemma}
\begin{proof}
In any product of $k$ matrix entries, some row occurs an odd number of times and, since $k<M$, another row does not occur. Left multiplication by the diagonal matrix that changes the signs of these two rows preserves Haar measure and negates the product.
\end{proof}

Despite the expected screening at each fixed odd matrix dimension, Chatterjee proved an area law in the strongly coupled large-$M$ limit of $\mrm{SO}(M)$ lattice Yang--Mills theory \cite{Chatterjee2019a}. The coupling convention in that limit is
\[
 \exp\!\left(M\lambda\sum_p\mrm{Tr}(Q_p)\right),
\]
with $\lambda$ fixed and small, so the parameter of Section \ref{section:Preliminaries} is $\beta=M\lambda$. The fixed-dimension strong coupling expansion and this large-dimension limit therefore need not retain the same screening contributions.

A related distinction is visible in the truncated model considered in \cite{CNS2025a}. Replace each exponential by
\[
 T_B(\beta\mrm{Tr}(Q_p)),\qquad T_B(t):=\sum_{k=0}^B\frac{t^k}{k!}.
\]
For even $B$, this weight is positive on the real line. If $M$ is odd and
\[
 1+2(d-1)B<M,
\]
then the truncated partition functions and simple Wilson-loop numerators are identical for $\mrm{SO}(M)$ and $\mrm{O}(M)$. Indeed, every edge matrix occurs to degree less than $M$. The even moments agree because $\mrm{O}(M)=\mrm{SO}(M)\sqcup(-\mrm{SO}(M))$, while the odd moments below $M$ vanish by the preceding lemma. Expanding the polynomials and integrating edge by edge proves the identity.

This indicates that in contrast to Theorem \ref{thm:Main-theorem}, the truncated model of \cite{CNS2025a} defined for $\mrm{SO}(M)$ would obey an area law, and moreover the author believes a very similar proof as provided in \cite{CNS2025a} would carry over in the $\mrm{SO}(M)$ setting.

\appendix
\section{Some $\mathrm{Haar}(\mathrm{SO}(3))$ expectation values}\label{Appendix:Expectation-Values}

In this section we record a few $\mathrm{Haar}(\mathrm{SO}(3))$ moment calculations which are used in the proof of Proposition \ref{lm:Tube-Lemma}.
\begin{lemma}\label{lm:Haar-SO(3)-Moments}
    Suppose $Q\sim \mathrm{Haar}(\mathrm{SO}(3))$. Then,
    \begin{equs}
        \E[Q_{ij}]=0,
    \end{equs}
    \begin{equs}
        \E[Q_{i_1j_1}Q_{i_2j_2}]=\frac{1}{3}\delta_{(i_1,j_1),(i_2,j_2)},
    \end{equs}
    and
    \begin{equs}
        \E[Q_{i_1j_1}Q_{i_2j_2}Q_{i_3j_3}]
        =\frac16\varepsilon_{i_1i_2i_3}\varepsilon_{j_1j_2j_3},
    \end{equs}
where $\varepsilon$ is the alternating tensor with $\varepsilon_{123}=1$.
\end{lemma}

\begin{proof}
Let $Q=(Q_{ij})_{1\le i,j\le 3}\sim \mathrm{Haar}(\mathrm{SO}(3))$.

We will repeatedly use the following fact: if
\[
D=\mathrm{diag}(\varepsilon_1,\varepsilon_2,\varepsilon_3)\in \mathrm{SO}(3),
\qquad \varepsilon_r\in\{\pm 1\},\quad \varepsilon_1\varepsilon_2\varepsilon_3=1,
\]
then $DQ$ and $QD$ have the same law as $Q$.  In particular, multiplying on the left flips the signs of two rows, and multiplying on the right flips the signs of two columns.

We will also use that Haar measure is invariant under left/right multiplication by signed permutation matrices of determinant one. In particular, even permutation matrices preserve Haar measure.

\medskip

\textbf{First moment.}
Fix $i,j$. Choose $D\in \mathrm{SO}(3)$ diagonal with $D_{ii}=-1$ (and hence exactly one other diagonal entry also equal to $-1$). Then
\[
(DQ)_{ij}=-Q_{ij}.
\]
Since $DQ\stackrel{d}=Q$, we get
\[
\E[Q_{ij}]=\E[(DQ)_{ij}]=-\E[Q_{ij}],
\]
hence
\[
\E[Q_{ij}]=0.
\]

\medskip

\textbf{Second moment.}
Let
\[
M_{i_1j_1,i_2j_2}:=\E[Q_{i_1j_1}Q_{i_2j_2}].
\]
We first show that this vanishes unless $i_1=i_2$ and $j_1=j_2$.

If $i_1\neq i_2$, choose $D=\mathrm{diag}(\varepsilon_1,\varepsilon_2,\varepsilon_3)\in \mathrm{SO}(3)$ with
$\varepsilon_{i_1}=-1$ and $\varepsilon_{i_2}=1$. Then
\[
(DQ)_{i_1j_1}(DQ)_{i_2j_2}
=
-\,Q_{i_1j_1}Q_{i_2j_2}.
\]
By invariance of law under left multiplication,
\[
M_{i_1j_1,i_2j_2}
=
-\,M_{i_1j_1,i_2j_2},
\]
so $M_{i_1j_1,i_2j_2}=0$. An identical argument with right invariance shows this expectation is also $0$ unless $j_1=j_2$.

Thus only the case $(i_1,j_1)=(i_2,j_2)$ can survive. By invariance under even row and column permutations, the value of $\E[Q_{ij}^2]$ does not depend on $i$ or $j$.

Now for any fixed row $i$,
\[
\sum_{j=1}^3 Q_{ij}^2=1
\]
since $Q$ is orthogonal. So taking expectations,
\[
\E[Q_{i_1j_1}Q_{i_2j_2}]
=
\frac13\,\delta_{i_1,i_2}\delta_{j_1,j_2}.
\]

\medskip

\textbf{Third moment.}
Set
\[
T_{i_1j_1,i_2j_2,i_3j_3}
:=
\E[Q_{i_1j_1}Q_{i_2j_2}Q_{i_3j_3}].
\]

We first determine when this can be nonzero. Suppose some row index appears twice or three times among $i_1,i_2,i_3$. Then one can choose
$D=\mathrm{diag}(\varepsilon_1,\varepsilon_2,\varepsilon_3)\in \mathrm{SO}(3)$ so that an odd number of the three factors
$(DQ)_{i_k j_k}$ change sign. Hence
\[
T_{i_1j_1,i_2j_2,i_3j_3}
=
-\,T_{i_1j_1,i_2j_2,i_3j_3},
\]
and therefore it is $0$.

So nonzero third moments require $i_1,i_2,i_3$ to be pairwise distinct. The same argument on the right shows that nonzero third moments also require $j_1,j_2,j_3$ to be pairwise distinct. Hence $T_{i_1j_1,i_2j_2,i_3j_3}$ can be nonzero only if both $(i_1,i_2,i_3)$ and $(j_1,j_2,j_3)$ are permutations of $(1,2,3)$. 

Now assume this is the case. Then there is a unique permutation $\sigma\in S_3$ such that
\[
j_k=\sigma(i_k),\qquad k=1,2,3.
\]
Reordering the three factors, it suffices to consider $\E[\prod_{i=1}^3Q_{i,\sigma(i)}]$. Set $c:=\E[Q_{11}Q_{22}Q_{33}]$. Choose signs $d_1,d_2,d_3\in\{\pm1\}$ with $d_1d_2d_3=\mrm{sgn}(\sigma)$, and let the $i$th column of $S$ be $d_i e_{\sigma(i)}$. Then $S\in\mrm{SO}(3)$, so right invariance gives
\[
 c=\E\!\left[\prod_{i=1}^3(QS)_{ii}\right]
   =\mrm{sgn}(\sigma)\,\E\!\left[\prod_{i=1}^3Q_{i,\sigma(i)}\right].
\]
Consequently,
\[
 T_{i_1j_1,i_2j_2,i_3j_3}=c\,\mrm{sgn}(\sigma).
\]

It remains to compute $c$. Using $\det Q=1$, taking expectations, and using the symmetry just established, each term has expectation equal to $c$ times its permutation sign. Since the determinant expansion already carries that sign, each of the six terms contributes $+c$ after expectation. Hence
\[
1=6c.
\]
Therefore
\[
\E[Q_{i_1j_1}Q_{i_2j_2}Q_{i_3j_3}]
=
\frac16\,\varepsilon_{i_1i_2i_3}\varepsilon_{j_1j_2j_3}.
\]
This is exactly the claimed formula.
\end{proof}

\begin{lemma}\label{lm:Haar-zero-integral}
    Fix a positive integer $k$ and suppose $Q_1,\dots,Q_k \sim \mrm{Haar}(\mrm{SO}(3))$ are all independent of each other. Then for any indices $i,j \in [3]$, subset of pairs $P \subset [k]^2$, and collection of nonnegative integers $\{\ell_{(a,b)}\}_{(a,b) \in P}$, we have
    \begin{equs}
        \E\bigg[(Q_1)_{i,j} \prod_{(a,b) \in P} \mrm{Tr}(Q_a Q_b^T)^{\ell_{(a,b)}}\bigg]=0
    \end{equs}
\end{lemma}
\begin{proof}
    Choose a diagonal matrix $D$ with $D_{jj}=-1$, one other diagonal entry equal to $-1$, and the remaining entry equal to $1$. Then $D \in \mrm{SO}(3)$, so by right invariance of the Haar measure,
    \begin{equs}            \E\bigg[(Q_1)_{i,j} \prod_{(a,b) \in P} \mrm{Tr}(Q_a Q_b^T)^{\ell_{(a,b)}}\bigg] &= \E\bigg[(Q_1 D)_{i,j} \prod_{(a,b) \in P} \mrm{Tr}(Q_a D D^T Q_b^T)^{\ell_{(a,b)}}\bigg]\\
    &=-\E\bigg[(Q_1)_{i,j} \prod_{(a,b) \in P} \mrm{Tr}(Q_a Q_b^T)^{\ell_{(a,b)}}\bigg]
    \end{equs}
    and hence $\E\bigg[(Q_1)_{i,j} \prod_{(a,b) \in P} \mrm{Tr}(Q_a Q_b^T)^{\ell_{(a,b)}}\bigg]=0$.
\end{proof}

\begin{lemma}\label{lm:double-edge-integration}
    For any matrices $A,B\in\R^{3\times3}$, if $Q \sim \mrm{Haar}(\mrm{SO}(3))$, then
    \begin{equs}
        \E[\mrm{Tr}(AQ)\mrm{Tr}(Q^*B)]=\frac{1}{3}\mrm{Tr}(AB)
    \end{equs}
\end{lemma}
\begin{proof}
    Expanding the traces and applying Lemma \ref{lm:Haar-SO(3)-Moments},
    \begin{equs}
       \E[\mrm{Tr}(AQ)\mrm{Tr}(Q^TB)]
       &=\sum_{i,j,k,l=1}^3 A_{ij}B_{kl}\,\E[Q_{ji}Q_{kl}]\\
       &=\frac13\sum_{i,j=1}^3A_{ij}B_{ji}
        =\frac13\mrm{Tr}(AB).
    \end{equs}
\end{proof}

\bibliographystyle{alpha}
\bibliography{references}

\end{document}